\documentclass[11pt]{article}
\usepackage{fullpage}
\usepackage{amsmath, amsthm, amscd,amssymb}
\usepackage[left=3cm, right=3cm, top=2cm]{geometry}
\usepackage[utf8]{inputenc}
\usepackage[all]{xy}
\usepackage{color}
\usepackage{comment}
\usepackage{appendix}
\usepackage{verbatim}
\usepackage{pdfsync}
\usepackage{hyperref}
\usepackage{cleveref}
\usepackage{float}
\usepackage{setspace}
\usepackage{xpatch}
\usepackage{subcaption}
\usepackage{pst-node}
\usepackage{tikz-cd}
\usetikzlibrary{positioning}
\usetikzlibrary{shapes,arrows}
\usepackage{tikz}
\usepackage{graphicx,epsfig}
\usepackage{mathrsfs}
\usepackage{verbatim}

\date{}

\let\oldproofname=\proofname
\renewcommand{\proofname}{\rm\bf{\oldproofname}}

\newtheorem{theorem}{Theorem}[section]

\newtheorem{lemma}[theorem]{Lemma}
\newtheorem{proposition}[theorem]{Proposition}

\theoremstyle{definition}

\newtheorem{remark}[theorem]{Remark}
\newtheorem{definition}[theorem]{Definition}

\newtheorem{claim}[theorem]{Claim}
\newtheorem{ques}[theorem]{Question}

\newtheorem{example}[theorem]{Example}

\def\1{\mathbf{1}}

\newcommand{\F}{\mathcal{F}}

\newcommand{\cL}{\mathcal{L}}

\newcommand{\M}{\mathcal{M}}

\newcommand{\C}{\mathcal{C}}
\renewcommand{\H}{\mathcal{H}}

\newcommand{\I}{\mathcal{I}}

\def\tH{\text{H}}
\def\tmes{\text{mes}}

\def\rNC{\mathrm{NC}}

\def\del{\mathrm{del}}
\def\lk{\mathrm{lk}}
\newcommand{\mes}{\textrm{mes}}

\definecolor{blue1}{rgb}{0.10,0.60,0.8}

\begin{document}
	\title{Bounds for the collapsibility number of a simplicial complex and non-cover complexes of hypergraphs }
	\author{ Rekha Santhanam\footnote{Department of Mathematics, Indian Institute of Technology Bombay, India. reksan@iitb.ac.in}, Samir Shukla\footnote{School of Mathematical and Statistical Sciences, Indian Institute of Technology Mandi, India. samir@iitmandi.ac.in}, Anurag Singh\footnote{Department of Mathematics, Indian Institute of Technology Bhilai, India. anurags@iitbhilai.ac.in} }
	\maketitle
	
	\begin{abstract}
   The collapsibility number of simplicial complexes was introduced by Wegner in order to understand the intersection patterns of convex sets. This number also plays an important role in a variety of Helly type results.  We show that the non-cover complex of a hypergraph $\mathcal{H}$  is  $|V(\mathcal{H)}|- \gamma_i(\mathcal{H})-1$-collapsible, where $\gamma_i(\mathcal{H})$ is the generalization of independence domination number of a  graph to hypergraph. This   extends the result of Choi, Kim and Park from graphs to hypergraphs. Moreover, the upper bound in terms of strong independence domination number  given by Kim and Kim for the Leray number of the non-cover complex of a hypergraph can be obtained as a special case of our result.
   
   In general, there can be a large gap between the collapsibility number of a complex and its well-known upper bounds. 
   In this article, we construct a sequence of upper bounds $\mathcal{M}_k(X)$ for the collapsibility number of a simplicial complex $X$, which lie in this gap. We also show that the bound given by $\mathcal{M}_k$ is tight if the underlying complex is $k$-vertex decomposable. 
    \end{abstract}

	\noindent {\bf Keywords} : collapsibility, non-cover complexes, hypergraphs 

	\noindent 2020 {\it Mathematics Subject Classification:} 05C69, 05E45, 52B22, 52A35
	
	\vspace{.1in}

	\section{Introduction}
Let $X$ be a (finite) simplicial complex. Let $\gamma,\sigma \in X$ be such that  $|\gamma|\leq d$ and $\sigma \in X$ is the only maximal simplex that contains $\gamma$. Then, $(\gamma, \sigma)$ is called a {\it free pair} and $\gamma$ is called a {\it free face} of $\sigma$ in $X$.   An  {\it elementary $d$-collapse} of $X$ is the simplicial complex $X'$ obtained from $X$ by
removing all those simplices $\tau $  of $X$ such that
$\gamma \subseteq \tau \subseteq \sigma$, and we denote this elementary  $d$-collapse by  $
X \xrightarrow{\gamma} X'$.
The complex $X$ is called \emph{$d$-collapsible} if there exists a sequence of elementary $d$-collapses 
$$
X=X_1\xrightarrow{\gamma_1} X_2 \xrightarrow{\gamma_2} \cdots \xrightarrow{\gamma_{k-1}} X_k=\emptyset
$$
from $X$ to the empty complex $\emptyset$. Note that every $d$-dimensional complex is always $d+1$ collapsible. Clearly, if $X$ is $d$-collapsible and $d < c$, then $X$ is $c$-collapsible. The \emph{collapsibility number} of $X$, denoted as $\C(X)$,  is the minimal integer $d$ such that $X $ is $d$-collapsible.

The notion of d-collapsibility of simplicial complexes was introduced by Wegner \cite{Wegner75}.  The motivation of introducing d-collapsibility comes from combinatorial geometry as a tool for studying intersection patterns of convex sets 
\cite{Eckhoff,Kalai2005,Tancer2013, Wegner75}.  The collapsibility number  plays an important role in the study of   various Helly type results  \cite{ AhroniHolzmanJiang2019, Kalai1984, Kalai2005}. The collapsibility number is also  related to an another well-studied combinatorial inviariant of a simplicial complex called {\it Leray number} $\mathcal{L}(X)$ (\Cref{Leray}).   In \cite{Wegner75}, Wegner proved that $\cL(X) \leq \C(X)$ for any simplicial complex $X$.

A well known bound for $\C(X)$ which is useful in proving several results of collapsibility is given by $d(X,\prec)$ (see \Cref{prop:minimal_exclusion}).  In \cite{Choi2020}, Choi et al. studied the collapsibility number of non cover complexes of graphs using this bound. 
They showed that the collapsibility number of non-cover complex  $\rNC(G)$ of a graph $G$  is bounded by $ |V(G)|- \gamma_i(G)-1$, where $\gamma_i(G)$ denotes the independence domination number of $G$ as defined in \cite{AharoniSzabo} (the authors in \cite{Choi2020} use the notation $i\gamma(G)$ for $\gamma_i(G)$).
One of the objectives of our article is to extend the main result of \cite{Choi2020} from graphs to hypergraphs.

Let $\H$ be  a hypergraph.  A set $B \subseteq V(\H)$ is called a {\it cover} of $\H$ if $e \cap B \neq \emptyset$ for any edge $e \in E(\H)$. A set $A \subseteq V(\H)$ is called a {\it non-cover} if it is not a cover of $\H$.  The {\it non-cover complex} $\rNC(\H)$ of $\H$ is a simplicial complex defined as 
$$
\rNC(\H) = \{ A \subseteq V(\H): A \ \text{is a non-cover of} \  \H \}.
$$
  In \Cref{thm:bound for nc}, we prove that for any hypergraph $\H$, 
  $\C(\rNC(\H))$ is bounded above by the $ |V(\H)|- \gamma_i(\H)-1$,  where $\gamma_i(\H)$  is a  generalization of the independence domination number of  graphs to hypergraphs (see \Cref{def:dcn} for the definition).  
  
In \cite{KimKim2021}, authors showed that $\cL(\rNC(\mathcal{H})) \leq |V(\mathcal{H})| - \gamma_{si}(\mathcal{H}) -1$ whenever $|e| \leq 2$ for every $e \in E(\mathcal{H})$ (where $\gamma_{si}(\mathcal{H})$ is the strong independence domination number of a hypergraph $\mathcal{H}$). In \Cref{lem:si and kappa}, we show that $\gamma_i(\H)=\gamma_{si}(\mathcal{H})$ whenever $|e| \leq 2$ for every $e \in E(\mathcal{H})$. Thus, our result, \Cref{thm:bound for nc}, is an improvement on their result.
In the same paper \cite[Theorem 1.6]{KimKim2021},  the authors give similar upper bounds for the Leray number of a hypergraph $\H$ in terms of  $\tilde{\gamma}(\mathcal{H})$, the strong total domination number (when $|e|\leq 3$)  and  $\gamma_E(\mathcal{H})$, the edgewise-domination number (see Definitions \ref{defn:strong total domination number} and  \ref{defn:edgewise-domination number}).  
In \Cref{counterexample}, we give a class of hypergraphs, for which $$\gamma_i(\H) > \max\{\left \lceil{\tilde{\gamma}(\mathcal{H})}/{2}\right\rceil, \gamma_E(\H)\}.$$  
This example also shows that the gap in the above inequality can be arbitrarily large.
Since $\cL(X) \leq \C(X)$ for any simplicial complex $X$, our result, \Cref{thm:bound for nc}  
implies a more general result and often gives a better bound.


  Even though $d(X,\prec)$ is better suited for theoretical arguments, as displayed earlier, there can be a substantial gap between the bound obtained by $d(X,\prec)$ and $C(X)$.  In \cite{TC22}, Biyiko{\u{g}}lu and Civan, introduced a combinatorial invariant $\M(G)$ for any graph $G$ and extended  \cite{civan} it to any simplicial complex $X$. It can be shown that  $\C(X) \leq \M(X)$ always holds.  
In this article, we prove that $d(X,\prec)$ is also an upper bound for $\M(X)$. In \Cref{example1}, we show that $\M(X)$ can be strictly greater than $\C(X)$. 
We give sharper bounds for collapsibility by introducing a new combinatorial notion  $\M_k(X)$, for a simplicial complex $X$ and each non-negative integer $k$, and show that   $\C(X) \leq \M_k(X) \leq \M_{k-1}(X) \leq \ldots  \leq \M_1(X) \leq \M_{0}(X) = \M(X) \leq d(X, \prec)$ (see \Cref{remark:theta} and \Cref{thetakprop}). We prove that for $k$-vertex decomposable simplicial complexes $X$, $ \C(X) = \M_k(X)$ (see \Cref{thm:kvdecomposable}). We also give an example of a complex $X$,  where $\C(X) = \M_1(X) < \M(X)$ (see \Cref{example1}). Given the recursive nature of the definition of $\M_k(X)$, we expect it to be better suited for computational purposes in comparison with the computations of $d(X,\prec)$ which depends on the ordering of maximal simplices.


\section{Non-cover complexes of hypergraphs}\label{section3}

A {\it hypergraph} $\H$ is an ordered  pair $(V(\H), E(\H))$, where $V$ is a (finite) set and  $E$ is a family of subsets of $V$. The elements of $V(\H)$ are called the vertices of $\H$, and the elements of $E(\H)$ its edges.  Let $\H$ be a hypergraph. Let $v \in V(\H)$. A vertex $w \in V(\H)$ is a {\it neighbour} of $v$, if there exists $e \in E(\H)$ such that $\{v, w\} \subseteq e$. The neighbour set of $v$ is defined as $N(v)=  \{w : w \text{ is a neighbour of  } v\}$. For $A \subseteq V(\H)$, the neighbour set of $A$ is $N(A) := \bigcup_{v \in A} N(v)$. For $S \subseteq V(\H)$, the induced subgraph $\H[S]$ is the hypergraph on the vertex set $S$ and any $e \subseteq S$ is an edge in $\H[S]$ if $e \in E(\H)$. A vertex $w$ is called isolated if $N(w)=\emptyset$.
A set $\I \subseteq V(\H)$ is called independent if $e \not\subseteq \I$ for all $e \in E(\H)$.

Let $A \subseteq V(\H)$.   A set  $W \subseteq V(\H) \setminus A$ is said to be a dominating set  of  $A$, if for any $v \in A$, there exists $w\in W$ and $e \in E(\H)$ such that 
$v, w \in e$, {\it i.e.,} $A \subseteq N(W)$. The domination number of $A$ is
$$
\gamma_A(\H) = \min \{|W| : W \ \text{is a dominating set of } A\}.
$$

\begin{definition}\label{def:dcn}
 Let $\H$ be hypergraph with no isolated vertex.   The \emph{independence domination number} of $\H$ is 
     $$\gamma_i(\H) = \max \{\gamma_{\I}(\H):  \I \subset V(\H) \text{ is independent}   \}.$$

 \end{definition} 





\begin{remark}\label{remark:definitionofdominationnumber}
Observe that a set $D$ is a cover of $\H$ if and only if $\overline{D}$ is an independent set of $\H$.  Further, $W$ is a dominating set of $A$, if $ A \subseteq N(W)$.
 Therefore, the following is easy to observe:
$$\gamma_i(\H) = \max \{\gamma_{\overline{D}}(\H):  D \ \text{is a cover of} \  \H \}.$$
 \end{remark}

The next result gives an upper bound for the collapsibility number of non-cover complexes of hypergraphs in terms of their independence domination number.

\begin{theorem}\label{thm:bound for nc}
Let $\H$ be a hypergraph with no isolated vertices. Then
\[C(\rNC(\H))\leq |V(\H)|- \gamma_i(\H)-1.\]
\end{theorem}

Before proving \Cref{thm:bound for nc}, we review the minimal exclusion principle, which will play a key role in the proof of \Cref{thm:bound for nc}.

Let $X$ be a simplicial complex on  a linearly ordered vertex set $V$ and let $\prec: \gamma_1,\ldots,\gamma_m$ be a  linear ordering  on the maximal simplices  of $X$.
Given a $\gamma \in X$,   the \textit{minimal exclusion sequence} $\mes(\gamma, \prec)$ of elements of $\gamma$ is defined as follows:

Let $j$ denote the smallest index such that $\gamma \subseteq \gamma_j$.
If $j=1$, then $\mes(\gamma, \prec)$ is the null sequence.
If $j\geq 2$, then $\mes(\gamma, \prec)=(v_1,\ldots, v_{j-1})$ is a finite sequence of length $j-1$ such that
$v_1=\min (\gamma\setminus \gamma_1)$ and  for each $k\in\{2, \ldots, j-1\}$, 
$$v_k=\begin{cases}
\min(\{v_1,\dots,v_{k-1}\}\cap (\gamma \setminus \gamma_k)) & \text{if } \{v_1,\dots,v_{k-1}\}\cap (\gamma \setminus \gamma_k)\neq\emptyset,\\
\min (\gamma\setminus \gamma_k) & \text{otherwise.}
\end{cases} $$

Let $M(\gamma, \prec)$ denote the set of vertices appearing in $\mes(\gamma, \prec)$. Define
$$ d(X, \prec):=\max_{\gamma \in X}|M(\gamma, \prec)|.$$

The following result gives us a bound for the collapsibility number of the complex $X$ using  $d(X, \prec)$.
\begin{proposition}\cite[Theorem 6]{Lew18} \label{prop:minimal_exclusion}
	If $\prec$ is a linear ordering of the maximal simplices of $X$, then $X$ is $d(X, \prec)$-collapsible.
\end{proposition}

For a positive integer  $n$, let $[n]$ denotes the ordered set $\{1, \ldots, n\}$. For   $A \subseteq [n]$, we let $\overline{A} = [n] \setminus A$. In the rest of the section, we assume that $\H$ be a hypergraph with no isolated vertices,  $V(\H) = [n]$ and if $e = \{v_1,v_2,\dots,v_r\}\in E(\H)$ then $v_1>v_{2}>\cdots>v_{r-1} > v_r$. Let $<_{L}$ be the lexicographic order on the set $E(\H)$. For example, let $\H$ be a hypergraph on vertex set $[4]$ and edge set $\{\{3,2,1\},\{4,3,1\},\{4,3,2\}\}$, then $\{3,2,1\}<_{L}\{4,3,1\}<_{L}\{4,3,2\}$. 

One can observe that every facet of $\rNC(\H)$ is the complement of an edge of $\H$.  We define a linear order $
\prec$ on the set of facets of $\rNC(\H)$ as follows:  $\gamma \prec \gamma'$ if $\overline{\gamma} <_{L} \overline{\gamma'}$. Let $\gamma_1 \prec \gamma_2 \prec \ldots \prec \gamma_m$ be all the facets of $\rNC(\H)$. 
Note that, for any $ 1\leq i < j \leq m$, $\max{\overline{\gamma_i}}\leq \max{\overline{\gamma_j}}$.


 From \Cref{remark:definitionofdominationnumber}, there exists a cover $D$ of $\H$ such that  $\gamma_{\overline{D}}(\H)= \gamma_i(\H)$. If $D$ is not a minimal cover, then there exists some $v \in D$ such that $D- \{v\}$ is still a cover. Then $\overline{D-\{v\}} = \overline{D} \cup \{v\}$ and clearly $\gamma_{\overline{D-\{v\}}}(\H) \geq \gamma_{\overline{D}}(\H)$. Therefore,  we can assume that $D$ is a minimal cover. Without loss of generality, let $D=\{1, \ldots, |{D}|\}$.


\begin{lemma}\label{lem:mes_equal}
Let $\gamma, \gamma' \in \rNC(\H)$. If $\overline{\gamma} \cap D = \overline{\gamma'} \cap D$ and the induced subgraph $\H[\overline{\gamma} \cap D]$ contains an edge, then 
$$\tmes(\gamma, \prec) = \tmes(\gamma', \prec).$$
\end{lemma}

\begin{proof}
Let $k$ be the smallest index such that $\gamma \subseteq \gamma_k$.  Since $\H[\overline{\gamma} \cap  D]$ contains an edge, let this edge be $\overline{\gamma_{t}}$ for some facet $\gamma_t$ of $\rNC(\H)$. 
Clearly, $t\geq k$ as $\gamma \nsubseteq \gamma_{i}$, {\it i.e.}, $\overline{\gamma_{i}}\nsubseteq \overline{\gamma}$ for any $i < k$. Since $\overline{\gamma_t}\subseteq D$, $\max(\overline{\gamma_t})<|{D}|$. 
Thus, $\gamma_i\prec \gamma_t$ (equivalently $\overline{\gamma_i}<_L \overline{\gamma_t}$) implies that $\max(\overline{\gamma_i})<|{D}|$ for all $i \leq t$, and therefore $\overline{\gamma_i}\subseteq D$ for all $i\leq t$. In particular, $\overline{\gamma_k}$ is an edge in $\H[\overline{\gamma} \cap  D]$.
It is given that $\overline{\gamma} \cap D = \overline{\gamma}'\cap D$, implying that $\gamma \cap D = \gamma'\cap D$. Hence, for each $i \in [k]$,
we get that 
$$\overline{\gamma_i}\cap \gamma =\overline{\gamma_i}\cap \gamma\cap D =\overline{\gamma_i}\cap \gamma' \cap D=\overline{\gamma_i}\cap \gamma'.$$

Therefore, we conclude that  $k$ is the smallest index such that $\overline{\gamma_k} \subseteq \gamma'$, {\it i.e.}, $\gamma' \subseteq \gamma_k$ and for every $i \in [k-1]$, the $i$th entry  of $\mes(\gamma', {\prec})$ is equal  to the $i^{th}$ entry of $\mes(\gamma, {\prec})$.
\end{proof}

\begin{lemma}\label{lem:neighbor_inequality}
For any $S \subseteq D$,
$$ |N(S) \cap \overline{D}| -|S| \leq |\overline{D}| - \gamma_{\overline{D}}(\H). $$
\end{lemma}

\begin{proof}
If $\overline{D} \subseteq N(S)$, then by definition of $\gamma_{\overline{D}}(\H)$, $|S| \geq \gamma_{\overline{D}}(\H)$ and result follows. 
So assume that  $\overline{D}\nsubseteq N(S)$. If $S =  D$, then $\overline{D} \subseteq N(S)$ and therefore assume  that $ D \setminus S \neq \emptyset$. 
Let $W$ be a minimal cardinality set such that $S \subsetneq W \subseteq D$ and $\overline{D} \subseteq N(W)$, {\it i.e.,} $W\setminus S$ is a minimal cover of $\overline{D}\setminus N(S)$. Then  $\gamma_{\overline{D}}(\H) \leq |W|$. Given any set $A$, it is clear that the cardinality of any minimal dominating set of $A$ is always less than or equal to the cardinality $A$. Therefore, 
$|W|- |S| \leq |\overline{D} | - |N(S) \cap \overline{D}| $, and hence the result follows.
\end{proof}

We now prove the  main result of this section. This is done by extending the idea of \cite{Choi2020} to hypergraphs.

\begin{proof}[Proof of \Cref{thm:bound for nc}]
 For a  $\sigma \in \rNC(\H)$,  we let $\Psi_{\sigma}=|N(\overline{\sigma}\cap D)\cap \overline{\sigma} \cap \overline{D}|$.

Let $\tau \in \rNC(\H)$. We first show that  for $v \in \tau  \cap \overline{D}$, if $v \in  M(\tau , \prec)$, then $v$ is a neighbour of some vertex in
$\overline{\tau } \cap  D$. Let $k$ be the smallest index such that the $k^{th}$ entry of $\mes(\tau , \prec)$ is $v$. Then $v \in \tau  \setminus \gamma_k$,
which implies that $v \in \overline{\gamma_k}$.  Since  $D$ is a cover and $\overline{\gamma_k}\in E(\H)$,  $\overline{\gamma_k} \cap D \neq \emptyset$. Choose a $w \in \overline{\gamma_k} \cap D$. Since $w \in D$  and $v \notin D, w < v$ (recall that $D=\{1,2,\dots,|D|\}$).  Further,  since $v$ is the $k^{th}$ entry of $\mes(\tau , \prec)$ and $w<v$, we get  $w \notin \tau $ which implies that $w \in \overline{\tau } \cap D$. Furthermore, since $v, w \in \overline{\gamma_k}$, $v$ is a neighbour of $w$.  Hence $ \tau  \cap \overline{D} \subseteq N(\overline{\tau} \cap D)$.

Therefore, 
\begin{align}\label{eq:mes}
|M(\tau , \prec)| &\le |\tau| = |\tau  \cap D|  + |\tau\cap \overline{D}| \nonumber \\
& = |\tau  \cap D|  + |N(\overline{\tau }\cap D)\cap (\tau  \cap \overline{D})| \nonumber \\
&= |D|-|\overline{\tau }\cap D| + |N(\overline{\tau }\cap D)\cap \overline{D}|-|N(\overline{\tau}\cap D)\cap \overline{\tau} \cap \overline{D}|\nonumber\\ 
&= |D|-|\overline{\tau }\cap D| + |N(\overline{\tau }\cap D)\cap \overline{D}|-\Psi_{\tau }\nonumber\\ 
 &\le |D|-\gamma_{\overline{D}}(\H)+|\overline{D}|-\Psi_{\tau } \nonumber\\
 &=|V(\H)|-\gamma_{\overline{D}}(\H)-\Psi_{\tau },
\end{align}
where the second equality follows from the fact that $ \tau  \cap \overline{D} \subseteq N(\overline{\tau} \cap D)$ and  last inequality holds by applying \Cref{lem:neighbor_inequality} to the set $\overline{\tau }\cap D$.

By \Cref{prop:minimal_exclusion}, it is sufficient to show that $\Psi_{\tau} \geq 1$. Suppose that $\Psi_{\tau }=0$. Since $\tau \in \rNC(\H)$, there exist an edge $e $ such that $e \subseteq \overline{\tau }$. If $e \cap \overline{D} \neq \emptyset$, then  $N(\overline{\tau }\cap D)\cap \overline{\tau } \cap \overline{D} \neq \emptyset$ (since $e\cap D\neq \emptyset$ as $D$ is a cover) and therefore $\Psi_{\tau } \geq 1$. 
Else, $e \subseteq D$ and  $\H[\overline{\tau} \cap D]$ has the edge $e$.
Let $\tau '=\tau \cap D$.
Then $\overline{\tau }\cap D=\overline{\tau'}\cap D$.
By \Cref{lem:mes_equal}, $\mes(\tau, \prec)=\mes(\tau', \prec)$ and therefore
$M(\tau , \prec)=M(\tau', \prec)$.
Note that $\Psi_{\tau'} = |N(\overline{\tau' }\cap D)\cap \overline{\tau'}\cap \overline{D}|= |N(\overline{\tau }\cap D)\cap \overline{D}|$ (since $\overline{D}\subseteq\overline{\tau'}$). We now show that $\Psi_{\tau'}\geq 1$. Recall that, $e\subseteq \overline{\tau}\cap D$. Let $v$ be a vertex in $e$. Then  $v \in \overline{\tau}\cap D$. Since $D$ is a minimal cover, there exists an edge $e'$ in $\H[\overline{D\setminus \{v\}}]$. Thus, $e'\setminus \{v\} \subseteq N(\overline{\tau }\cap D)\cap \overline{D}$ implying that $\Psi_{\tau'}\geq 1$.

Thus, by replacing $\tau $ by $\tau'$ and using \Cref{eq:mes},  we conclude that $|M(\tau, \prec)|  \leq  |V(\H)|-\gamma_i(\H)-1$. Hence, the definition of $d(X, \prec)$ along with \Cref{prop:minimal_exclusion} implies the following. 
\begin{equation}\label{equation:mainresult}
    C(\rNC(\H))\leq d(\rNC(\H), \prec) \leq |V(\H)|- \gamma_i(\H)-1.
\end{equation}
This completes the proof of \Cref{thm:bound for nc}.
\end{proof}






We now compare \Cref{thm:bound for nc} with the results of Kim and Kim \cite{KimKim2021}, where they established upper bounds on the Leray number of $\rNC(\H)$ with various domination parameters of the hypergraphs $\H$. To do this comparison, we first recall the required terminology from \cite{KimKim2021}.

Let $\mathcal{H}$ be a hypergraph.
Let $v \in V(\H)$  and $B$ be a subset of $V(\H)$. Then 
 $B$ {\em strongly totally dominates} $v$ if there exists $B' \subseteq B \setminus \{v\}$ such that $B' \cup \{v\} \in E(\H)$.  

	Let $W$ be a subset of $V(\H)$.
	If $B \subseteq V$ strongly totally dominates every vertex in $W$, then   $B$ is said to be  {\em strongly dominates} $W$.

	The {\em strong total domination number of $W$ in $\mathcal{H}$} is defined as 
\[\gamma(\mathcal{H}; W) := \min\{|B|:B\subseteq V(\H), ~B\text{ strongly dominates } W\}.\]

\begin{definition}\label{defn:strong total domination number}
	The {\em strong total domination number} $\tilde{\gamma}(\mathcal{H})$ of $\mathcal{H}$ is the strong total domination number of $V(\H)$, {\it i.e.}, $\tilde{\gamma}(\mathcal{H}) = \gamma(\mathcal{H}; V(\H))$.
\end{definition}
	
A set 	$ \I \subseteq V(\H)$ is said to be {\em strongly independent} in $\mathcal{H}$ if it is independent and every edge of $\mathcal{H}$ contains at most one vertex of $\I$.	
\begin{definition}\label{defn:strong independence domination number}
The {\em strong independence domination number} of $\mathcal{H}$ is the integer
\[\gamma_{si}(\mathcal{H}) := \max\{\gamma(\mathcal{H}; \I): \I \text{ is a strongly independent set of }\mathcal{H}\}.\]
\end{definition}

\begin{definition}\label{defn:edgewise-domination number}
	The {\em edgewise-domination number} of $\mathcal{H}$ is the minimum number of edges whose union strongly dominates the $V(\H)$, i.e.
	\[\gamma_{E}(\mathcal{H}) := \min\{|\F|: \F \subseteq E(\H), \bigcup_{e \in \mathcal{F}} e\text{ strongly dominates }V(\H)\}.\]

\end{definition}

	\begin{theorem}\cite[Theorem 1.6]{KimKim2021} \label{leray numbers}
	Let $\mathcal{H}$ be a hypergraph with no isolated vertices.
	Then 
	\begin{enumerate}
	\item[(i)] If $|e| \leq 3$ for every $e \in E(\mathcal{H})$, then $L(\rNC(\mathcal{H})) \leq |V(\mathcal{H})| - \left \lceil\frac{\tilde{\gamma}(\mathcal{H})}{2}\right\rceil -1$.
	\item[(ii)] If $|e| \leq 2$ for every $e \in E(\mathcal{H})$, then $L(\rNC(\mathcal{H})) \leq |V(\mathcal{H})| - \gamma_{si}(\mathcal{H}) -1$.
	\item[(iii)] $\cL(\rNC(\mathcal{H})) \leq |V(\mathcal{H})| - \gamma_{E}(\mathcal{H}) -1$.
	\end{enumerate}
	\end{theorem}

The following lemma shows that  \Cref{leray numbers} $(ii)$ is a special case of  \Cref{thm:bound for nc}.
 
\begin{lemma}\label{lem:si and kappa}
Let $\H$ be a hypergraph. If $|e| \leq 2$ for all $e \in E(\H)$, then $\gamma_i(\H) = \gamma_{si}(\H)$.
\end{lemma}

\begin{proof}

Since every strongly independent set is independent, it is clear from definitions of $\gamma_i(\H)$ and $\gamma_{si}(\H)$ that $\gamma_i(\H) \geq \gamma_{si}(\H)$. We now show that $\gamma_{si}(\H) \geq \gamma_i(\H)$.

Let $D$ be a cover of $\H$ such that $\gamma_{\overline{D}}(\H) = \gamma_i(\H)$. Since $D$ is a cover, for each cardinality one edge $\{x\}$, $x \in D$. Further, since $|e| \leq 2$ for all $e \in E(\H)$,  we see that   $\overline{D}$ is an strongly independent set. Let $S \subseteq D$ such that $|S| = \gamma_i(\H)$ and $\overline{D}  \subseteq N(S)$. Then $S$  will be of minimal cardinality which strongly dominates $\overline{D}$. Hence $\gamma_{si}(\H) \geq |S| = \gamma_i(\H)$. 
\end{proof}

We now give an example of a class of hypergraphs for which the difference $\gamma_i(\H) - \max\{\left \lceil{\tilde{\gamma}(\mathcal{H})}/{2}\right\rceil, \gamma_E(\H)\}$ can be made arbitrarily large.

\begin{example}\label{counterexample}
Let $n \geq 2$ and  let $H_1, H_2, \ldots, H_n$ be $n$ distinct  star graphs, where center vertex of the graph $H_i$ is $a_i$ for $1 \leq i \leq n$. Let $\H$ be hyper graph on vertex set $V(H_1) \cup \ldots \cup V(H_n)$ and edge set $E(\H) = E(H_1) \cup \ldots \cup E(H_n) \cup \{\{a_i, a_{i+1}\} : 1 \leq i \leq n-1\} \cup \{\{a_1, \ldots, a_n\}\}$. 
Then it is easy to check that  $\I = V(\H) \setminus \{a_1, \ldots, a_n\} $ is an independent set and $\{a_1, \ldots, a_n\}$ is the minimum dominating set of $\I$. Therefore $\gamma_{\I}(\H) = n$ and $\gamma_i(\H) \geq n$. 
 Since the set $\{a_1, \ldots, a_n\}$ strongly dominates $V(\H)$ and it is an edge in $\H$, we conclude that $\gamma_E(\H) = 1$ and $\tilde{\gamma}(\H) \leq n$. 
\end{example}


\section{The $\M_k$ number of a complex}

Tancer \cite{Tancerstrongd} showed that the collapsibility number of a simplicial complex is bounded by the collapsibility number of link and deletion of $X$ with respect to any vertex $v$.  This allows for inductive arguments to find the bounds on collapsibility number of a simplicial complex \cite{Lew18}. 

For any simplicial complex $X$ and $\sigma \in X$, the subcomplexes \emph{link} and \emph{deletion} of $\sigma$ in $X$ are defined as follows
	\begin{equation*}
	\begin{split}
	\mathrm{lk}(\sigma,X) & = \{\tau \in X : \sigma \cap \tau = \emptyset,~ \sigma \cup \tau \in X\}, \\
	\mathrm{del}(\sigma,X) & = \{\tau \in X : \sigma \nsubseteq \tau\}.
	\end{split}
	\end{equation*}

 Biyiko{\u{g}}lu and Civan \cite{TC22} defined $\M(G)$ inductively for any graph $G$ and then extended it for any simplicial complex $X$ as follows (\cite{civan}). 
 \begin{definition}
     Let $X^o$ denote the set of vertices $v$ in $X$ such that $\lk(v, X) \neq \del(v, X)$. Define 

\begin{equation*}
\M(X)=\begin{cases}
0 & \text{if } X^o=\emptyset,\\
 \min\limits_{v \in X^0}\{\max \{\M(\lk(v, X))+1 , \M(\del(v, X)\}\} & \text{otherwise}.
\end{cases} \end{equation*}

 
 \end{definition}

They \cite{civan} showed   that $\C(X) \leq \M(X)$ for all $X$.  In this article, we  introduce a sequence of invariants $\M_k(X)$ which lie between $\C(X)$ and $\M(X)$, where $\M_0(X) = \M(X)$ and show that $\C(X) \leq \M_k(X)$ for each $k \geq 0$ (see \Cref{thetakprop}).  

We first introduce the notation  we use in the rest of  the paper. Let $X$ be a simplicial complex. We denote the set of vertices of $X$ by $V(X)$.  For $A \subseteq V(X)$, the induced subcomplex on  vertex set $A$ is  $X[A] = \{\sigma \in X : \sigma \subseteq A\}$. For  $k \geq 0$, we let $X_{(k)}$ denote the set of $k$-dimensional faces of $X$ and  
$$X_{(k)}^o=\{ \sigma \in X_{(k)} : \lk{(\sigma, X)} \neq X[(V(X)\setminus V(\sigma)] \}.$$ 

\begin{lemma}\label{lemma:conesimplex}
Let $X$ be a simplicial complex of dimension at least $k$. If $X_{(k)}^o = \emptyset$, then $X$ is a simplex.
\end{lemma}
\begin{proof}
 If $k = 0$, then we prove the result by the induction on the number of vertices of $X$. The base case ({\it i.e.,} $X$ is a vertex) is trivially true. Now observe that, 
if $X_{(0)}^o = \emptyset$ then $\lk(v, X)_{(0)}^o = \emptyset$. Moreover, for any vertex $v$, $\lk(v, X) = X - \{v\}$, which implies that $ X = (X - \{v\}) \ast \{v\}$.  Hence by induction on the number of vertices, we get that $\lk(v, X)$ is simplex and therefore $X$ is a simplex. 

Let $k > 0$.   Let $\sigma \in X$ be a $k$-dimensional simplex. Then $\lk(\sigma, X) = X[V(X) \setminus \sigma]$. Hence $X = \lk(\sigma, X) \ast \sigma$. Let $Y = \lk(\sigma, X)$. If $Y_{0}^o=\emptyset $, then $Y$ is a simplex and therefore $X$ is a simplex. If  $Y_{(0)}^o \neq \emptyset$, then $\lk(v, Y) \neq Y - \{v\}$ for some $v \in V(Y)$. Choose $w \in \sigma$ and let $\tau = (\sigma \setminus \{w\}) \cup \{v\}$. Then $\lk(\tau, X) \neq X[V(X) \setminus \tau]$, a contradiction. Hence $Y_{(0)}^o = \emptyset$. By induction $Y$ is a simplex and therefore 
$X = Y \ast \sigma$ implies that $X$ is a simplex.
 \end{proof}

\begin{definition}\label{definition:theta_k}
Let $X$ be simplicial complex and let $k$ be a non negative integer. Define $\M_0(X) = \M'_0(X)= \M(X)$ and for $k \geq 1$, define $\M_k$ inductively as follows;
\begin{equation*}
\M_k'(X)=\begin{cases}
\M_{k-1}(X) & \text{if } X_{(k)}^o=\emptyset,\\
\min\limits_{\sigma\in X_{(k)}^o}\{ \max \{ \M'_k(\del{(\sigma, X))},   \M'_k(\lk{(\sigma, X))}+k+1\} & \text{otherwise},
\end{cases} \end{equation*}

and $\M_k(X) = \min\{\M'_k(X), \M_{k-1}(X)\}$.
\end{definition}

\begin{remark}\label{remark:theta}
Note by definition $\M_k(X) \leq \M_{k-1}(X)$ for all $k\geq 1$.
\end{remark}

We now give an example where $\M_1 < \M_0$.

\begin{example} (Example V6F10-6 from \cite{MorTak}) \label{example1}
Let $\Delta$ be the simplicial complex on the vertex set $\{1, 2, 3, 4, 5,6\} $ with the set of  facets
$$\{\{1, 2, 3\}, \{1, 2, 4\}, \{1, 2, 5\}, \{1, 3, 4\},\{1, 3, 6\}, \{2, 4,5\}, \{2,5,6\}, \{3,4,6\},\{3, 5,6\}, \{4,5,6\}\}.$$

This example was also discussed in \cite{Dochvdecomps} as an example of a complex which is $1$-vertex decomposable but not $0$-vertex decomposable. Thus, from \Cref{thm:kvdecomposable}, we get that $\C(\Delta)=\M_1(\Delta)$. We now show that $\M_1(\Delta)\leq 2<3\leq \M_0(\Delta)$.

To compute $\M_1(\Delta)$, let us look at $\M_1(\lk(\{1,5\},\Delta))$ and $\M_1(\del(\{1,5\},\Delta))$. Observe that $\lk(\{1,5\},\Delta)$ is a point $\{2\}$ implying that $\M_1(\lk(\{1,5\},\Delta))=0$. Thus, $$\M_1(\Delta)\leq\max\{\M_1(\del(\{1,5\},\Delta)), 2\}.$$
Here, the set of  facets of $\del(\{1,5\},\Delta))$ is  $$\{\{1,2,3\}, \{1,2,4\}, \{1,3,4\},\{1,3,6\}, \{2,4,5\}, \{2,5,6\},\{3,4,6\},\{3,5,6\}, \{4,5,6\}\}.$$ It is easy to verify by doing a similar calculation on the deletion complexes using the sequence $\{\{1,6\}, \{2,3\},\{1,3\},\{1,2\},\{2,4\},\{2,5\},\{3,4\},\{3,5\},\{4,5\}\}$ of $1$-faces, link complex at every step is a simplex and the deletion complex at the end is $1$-dimensional. Observe that $\C(X)$ is always less than or equal to the dimension of $X$. Hence by using the sequence  $\{\{1,6\}, \{2,3\},\{1,3\},\{1,2\},\{2,4\},\{2,5\},\{3,4\},\{3,5\},\{4,5\}\}$ of $1$-faces in the $\del(\{1, 5\}, \Delta)$, we conclude that $\M_1(\Delta) \leq 2$. 

Observe that the link of every vertex contains an induced subcomplex isomorphic to a triangulation of a circle. Hence the collapsibility number of link of every vertex is $2$.  Thus for each  vertex $v\in \Delta$, $\M_0(\lk(v,\Delta))\ge \C(\lk(v,\Delta))\geq 2$. Hence by definition $\M_0(\Delta)\geq 3$. 
\end{example}


\begin{lemma}[{\cite[Proposition 1.2]{Tancerstrongd}}]\label{lem:tancer}

Let $X$ be a simplicial complex and let $v$ be a vertex of $X$. Then $\C(X) \leq \max\{\C(\del{X, v}), \C(\lk(v, X))+1\}$. 

\end{lemma}

\begin{theorem}\label{thetakprop}
Let $X$ be a simplicial complex. Then for any $k \geq 0$, 
 $$ \C(X)\le \M_k(X).$$
\end{theorem}
\begin{proof}
Proof is by induction on $k$. If $k=0$ and $X_{(0)}^o=\emptyset$, then $X$ is a simplex and $C(X)=0=\M_0(X)$.
If $k=0$ and $X_{(0)}^o\neq\emptyset$, then the result follows from \Cref{lem:tancer} and the definition of $\M_0(X)$.
Let $k \geq 1$ and assume that 
$\C(X) \leq \M_r(X)$ for $0 \leq r < k$.  We now prove that $\C(X) \le \M_k(X)$ by induction on the number of $k$-simplices of $X$. If $X$ has no $k$-simplex then clearly $X_{(k)}^o=\emptyset$ implying that $\M_k(X)=\M_{k-1}(X)$ and hence the result follows. 

By definition, $\M_k(X)=\min\{\M_{k-1}(X),\M_k'(X)\}$. If $\M_k(X)=\M_{k-1}(X)$ then the result follows from induction. Now assume that $\M_k(X)=\M_k'(X)$.

We first prove a generalization of \Cref{lem:tancer}. 
\begin{claim}\label{claim}
For any $\sigma \in X_{(k)}$ ({\it i.e.}, $\sigma$ is a $k$-face) $$\C(X) \le \max\{ \C(\del(\sigma, X), \C(\lk(\sigma, X)) +k+1 \}.$$
\end{claim}
\begin{proof}[Proof of \Cref{claim}]Let $\lk(\sigma, X)$ be $d$-collapsible. Then there exist a sequence of elementary $d$-collapses such that 

$$ \lk(\sigma, X)=X_0 \xrightarrow{\sigma_1}  X_1 \xrightarrow{\sigma_2} X_2 \ldots \xrightarrow{\sigma_r} X_r=\emptyset.$$
Since $\lk(\sigma, X) \xrightarrow{\sigma_1}  X_1$ is an elementary collapse, there exist a facet $\tau_1\in \lk(\sigma, X)$ such that $\sigma_1$ is a free face of $\tau_1$ in $\lk(\sigma, X)$. Therefore, $\sigma_1 \cup \sigma$ is a free face of $\tau_1\cup \sigma$ in $X$. Furthermore, since $|\sigma_1 \cup \sigma|\le d+k+1$, we get an elementary $(d+k+1)$-collapse in $X$. Hence, the sequence 
$$ X=Y_0 \xrightarrow{\sigma_1\cup \sigma }  Y_1 \xrightarrow{\sigma_2\cup \sigma } Y_2 \ldots \xrightarrow{\sigma_r\cup \sigma} Y_r=\del(\sigma, X)$$
gives a sequence of elementary $(d+k+1)$-collapses of $X$ onto $\del(\sigma, X)$.
This implies that the collapsibility number of $X$ is less than or equal to $\max\{ \C(\del(\sigma, X)), d+k+1\}$. 
\end{proof}

If $X_{(k)}^o= \emptyset$, then the result follows from the induction on $k$ (since $\M_k(X)=\M_{k-1}(X)$). For $X_{(k)}^o\neq  \emptyset$, let $\sigma\in X_{(k)}$ such that $\M_k'(X)=\max\{\M_k'(\del(\sigma, X)), \M_k'(\lk(\sigma,X))+k+1\}$.
From the previous claim, we have that 
\begin{equation*}
\begin{split}
    \C(X) & \le \max\{ \C(\del(\sigma, X), \C(\lk(\sigma, X)) +k+1 \}\\
& \le \max\{\M_k(\del(\sigma, X)), \M_k(\lk(\sigma,X))+k+1\}\\
& \leq \max\{\M_k'(\del(\sigma, X)), \M_k'(\lk(\sigma,X))+k+1\}\\
&=\M_k'(X)=\M_k(X).
\end{split}
\end{equation*}
Here, the second inequality follows from induction, and the third inequality follows from the fact that $\M_k(X)\leq \M_k'(X)$.
\end{proof}

\begin{remark} Note  that \Cref{thetakprop}, along with \Cref{example1}, implies that $\M_1$ is a better approximation to $\C(X)$ than $\M_{0}$. 
\end{remark}

In our next result, we show that the bound obtained in \Cref{thetakprop} is tight for a particular class of complexes known as $k$-vertex decomposable complexes. Given a simplicial complex  $X$ its  \emph{pure $n$-skeleton},  $X^{[n]}$ is  the subcomplex of $X$ spanned by all  $n$-faces of $X$. The complex $X$ is said to be \emph{ pure} $n$-dimensional complex if $X=X^{[n]}$.

A pure  $d$-dimensional simplicial complex $X$ is said to be {\it shellable},
if its maximal simplices can be ordered $\Gamma_1, \Gamma_2 \ldots, \Gamma_t$ in such a way
that the subcomplex $(\bigcup\limits_{i = 1}^{k-1} \Gamma_i) \cap \Gamma_k$ is pure
and $(d-1)$-dimensional for all $k = 2, \ldots, t$.  A pure simplicial complex $X$ is said to be \emph{Cohen Macaulay} if, for all simplices $\sigma \in X$, the complex $\lk(\sigma,X)$ is homologically ($\text{dim}(\lk(\sigma,X))-1$)-connected, {\it i.e.}, $\tilde{H}_i(\lk(\sigma,X))=0$ for all $i < \text{dim}(\lk(\sigma,X) $. As a consequence, we get that if $X$ is Cohen Macaulay, then $\lk(\sigma, X)$ is also Cohen Macaulay for any $\sigma \in X$. 

Alternatively, a pure simplicial complex $X$ is said to be \emph{Cohen Macaulay} if each induced subcomplex $A$ of $X$ is homologically ($\text{dim}(A)-1$)-connected. 
From this definition and standard facts on homology, it can be easily verified that if $X$ is Cohen Macaulay, then  any skeleton of $X$ is also Cohen Macaulay.

\begin{definition} \cite[Definition 5.1]{Dochvdecomps}
For  $k \geq 0$, a pure  $r$-dimensional simplicial complex  $X$ is said to be {\it $k$-vertex decomposable} if $X$ is a simplex or $X$ contains a face $\sigma$ such that 
\begin{enumerate}
    \item $\text{dim}(\sigma) \leq k $.
\item both $\del(\sigma, X)$ and $\lk(\sigma, X)$ are $k$-vertex decomposable, and 
\item $\del(\sigma, X)$ is pure and the dimension is same as that of $X$.  (Such a face $\sigma$ is called a {\it shedding} face of $X$). 
\end{enumerate}
\end{definition}

The $k$-vertex decomposability ($k \geq 1$) of a complex interpolates between the shellability  and $0$-vertex decomposability of the complex. More precisely,
$$
0\text{-vertex decomp.} \implies k\text{-vertex decomp.}  \implies \text{shellability} \implies \text{Cohen Macaulay}.
$$
The first two implications are discussed in  \cite[Section 5]{Dochvdecomps}. The last implication follows from \cite[Section 11]{bjorner}.

\begin{theorem}\label{thm:kvdecomposable}

   If $X$ is $k$-vertex decomposable for some $k\geq 0$, then $$\C(X)=\M_k(X)=\M'_k(X).$$

\end{theorem}

To prove \Cref{thm:kvdecomposable}, we need a few results which we prove now.
Recall that $X_{(k)}^o=\{ \sigma \in X_{(k)} : \lk{(\sigma, X)} \neq X[(V(X)\setminus V(\sigma)] \}$. 

\begin{lemma}\label{lem:existance_of_nonconeface}
Let $X$ be of a $k$-vertex decomposable simplicial complex of dimension at least $k$. Then $X_{(k)}^o \neq \emptyset$.
\end{lemma}
\begin{proof}
 Let $X$ be $r$-dimensional, where $r\ge k$. Since $X$ is $k$-vertex decomposable, there exists a shedding face $\tau \in X$ such that $\dim(\tau)\leq k$ and dim$(\del(\tau, X))=r$. Since $X$ is pure and $r \geq k$, there exists a $k$-face $\sigma \in X$ such that $\tau \subseteq \sigma$.  We now prove that $\sigma \in X_{(k)}^o$. On contrary, assume that $\sigma \notin X_{(k)}^o$, {\it i.e.},  $\lk{(\sigma, X)} = X[(V(X)\setminus V(\sigma)]$. Let $\gamma$ be a facet of $X[(V(X)\setminus V(\sigma)]$. This implies that $ \gamma\sqcup \sigma$ is a facet of $X$. Hence, $(\gamma\cup\sigma)\setminus \tau$ is a facet of $\del(\tau,X)$. Now observe that dim$((\gamma\cup\sigma)\setminus \tau)< \dim(\gamma\cup \sigma)= r$. This contradicts the fact that $\del(\tau,X)$ is pure and of dimension $r$. Hence $\sigma \in X_{(k)}^o$.
\end{proof}

The following proposition is a generalization of \cite[Theorem 4.2]{LerayResult}.

\begin{lemma} \label{thm-MayerVietoris}
Let $Y$ be a simplicial complex and suppose that $\tilde{H}_{n-k}(\lk(\sigma,Y)) \neq 0$ for a $k$-face $\sigma \in Y$. If $\lk(\sigma,Y)^{[n-k]}$ is contained in a subcomplex $Y_0$ of $\del(\sigma,Y)$ with $\tilde{H}_{n-k}(Y_0) =0$, then $\tilde{H}_{n+1}(Y) \neq 0$. 
\end{lemma}
\begin{proof}

Given a $k$-face $\sigma$, let st$(\sigma, Y)=\{\tau \in Y : \sigma \subseteq \tau\}$. Then, $Y=\del(\sigma, Y) \cup \text{st}(\sigma, Y) $ and $\del(\sigma, Y) \cap \text{st}(\sigma, Y) =\lk( \sigma, Y) \ast \partial(\sigma)$, where $\partial(\sigma)= \{\tau : \tau \subsetneq \sigma\}$.

Using Mayer-Vietoris we get that the sequence 
$$ \tilde{H}_{n+1}(Y) \to \tilde{H}_{n}(\lk(\sigma,Y)\ast \partial(\sigma)) \xrightarrow{i} \tilde{H}_{n}(\del(\sigma,X))  $$
is exact. Note that, 
$$\tilde{H}_n(\lk(\sigma,Y)\ast \partial(\sigma))\cong \tilde{H}_n(\Sigma^{k}(\lk(\sigma,Y))) \cong \tilde{H}_{n-k}(\lk(\sigma,Y)).$$ 

Since the map $i$ is induced by an inclusion of $\lk(\sigma,Y)^{[n-k]}$ in $ Y_0\subseteq\del(\sigma,Y)$, the map $i$ is trivial. Thus, by the exactness of the above diagram, we get the required result. 
\end{proof}


\begin{lemma} \label{SCM}
Let $k\geq 1$ be a positive integer, and $Y$ be a simplicial complex. If $\tau$ is a shedding $k$-face for $Y$, $\del( \tau,Y)$ is  Cohen Macaulay and $\tilde{H}_{n-k}(\lk(\tau,Y)) \neq 0$, then $\tilde{H}_{n+1}(Y) \neq 0$. 
\end{lemma}
\begin{proof}
If $\tau$ is a shedding $k$-face for $Y$, then $\del(\tau,Y)$ is a pure complex of $\text{dim}(Y)$. Moreover, $\lk(\tau,Y)^{[n-k]}$ will be contained in the subcomplex $\del(\tau,Y)^{[n-k]} \subseteq \del(\tau,Y)^{[n+1]}$. Further,  $\del(\tau,Y)$ is a Cohen Macaulay complex  implies that $\del(\tau,Y)^{[n]}$ is Cohen Macaulay for all $n\ge 1$. In particular, choosing $\emptyset=\sigma \in  \del( \tau,Y)$ we get that $\del( \tau, Y)^{[n+1]}=\lk(\sigma,\del(\tau,Y)^{[n+1]})$ is homologically $n$-connected. 
Therefore,   $\tilde{H}_{n+1}(Y)$ $ \neq 0$ by \Cref{thm-MayerVietoris}.
\end{proof}

Our next result establishes the commutativity of link and deletion of disjoint faces in a complex.

\begin{lemma}\label{lem:linkdeletionface}
Let $\sigma, \tau \in X$ such that $\sigma \cap \tau =\emptyset$. Then, $\lk(\tau,(\del(\sigma,X))=\del(\sigma,\lk(\tau,X))$.
\end{lemma}
\begin{proof}
Let $\gamma\in \lk((\tau,\del(\sigma,X))$. Thus $\gamma \cup \tau \in \del(\sigma,X)$ implying that $\sigma \nsubseteq (\gamma \cup \tau)$. 
This gives us that $\sigma \nsubseteq \gamma$. Moreover, we know that $\gamma \in \lk(\tau,\del(\sigma,X)) \subseteq \lk(\tau,X)$. Therefore, the last two statements imply that $\gamma \in \del(\sigma, \lk(\tau, x))$. 

Now let $\eta \in \del(\sigma,\lk(\tau,X))$. So, $\sigma \nsubseteq \eta$. Furthermore, $\sigma \cap \tau = \emptyset$ implies that $\sigma \nsubseteq \eta \cup \tau$. Hence $\eta\cup \tau \in \del(\sigma,X)$ which gives us that $\eta \in \lk(\tau,\del( \sigma,X))$.
\end{proof}

\begin{definition}\label{Leray}
A simplicial complex $X$ is called  {\em $k$-Leray} if $\tilde{H}_i(\mathsf{L}) = 0 $ for all $i\geq k$ and  for every induced subcomplex $\mathsf{L}\subseteq X$. The Leray number $\cL(X)$ of $X$ is the least integer $k$ for which $X$ is $k$-Leray.
\end{definition}

\begin{proposition}\label{sheddingface} If $\sigma$ is a shedding $k$-face for a simplicial complex $X$ such that $\del(\sigma,X)$ is Cohen-Macaulay, then $\cL(X)\geq \max\{\cL(\del(\sigma,X)), \cL(\lk(\sigma,X))+k+1\}$.    
\end{proposition}
\begin{proof}
The proof for the case $k=0$ follows from \cite[Theorem 1.5]{LerayResult}.

By \cite[Lemma 2.3]{LerayResult}, $\cL(X)\ge d$ if and only if $\tilde{\tH}_{d-1}(\lk(\gamma,X)) \neq 0$ for some  $\gamma \in X$.
Let $\cL(\lk(\sigma,X))=d$, then there exists a face $\tau \in \lk(\sigma,X)$ such that $\tilde{H}_{d-1}(\lk(\tau,\lk(\sigma,X) )) \neq 0$.

Since $\tau \cap \sigma =\emptyset$, by \Cref{lem:linkdeletionface}, $\lk(\tau,(\del(\sigma,X))=\del(\sigma,\lk(\tau,X))$. Furthermore, since the link of any simplex in a Cohen-Macaulay complex is again Cohen-Macaulay, the complex $\lk(\tau,(\del(\sigma,X))=\del(\sigma, \lk(\tau,X))$ is Cohen-Macaulay. Since the link of any face in a pure complex is again pure, it is easy to check that $\sigma$ is a shedding face for $\lk(\tau,X)$ as well. Since $\tilde{H}_{d-1}(\lk( \sigma,\lk(\tau,X)) \neq 0$, by \Cref{SCM}, we get that $\tilde{H}_{d+k}(\lk( \tau,X)) \neq 0$. This implies that $\cL(X)\geq d+k+1$. 

Now it is sufficient to prove that $\cL(X) \geq \cL(\del(\sigma, X))$. The proof is by induction on number of vertices in $X$. Let $Y = \del(\sigma, X) $. Let $A \subseteq V(Y)$. Then observe that  $Y[A] = \del(\sigma \cap A, X [A])$. By induction, $\cL(X[A]) \geq \cL(Y[A])$. Since $\cL(X) \geq \cL(X[A])$ by taking $A = V(Y)$, we get that 
$$
\cL(X) \geq \cL(X[A]) \geq \cL(Y[A]) = \cL(Y). 
$$
\end{proof}

We can now  prove \Cref{thm:kvdecomposable}.
 
\begin{proof}[Proof of \Cref{thm:kvdecomposable}]
We know that $\cL(X) \leq \C(X) \leq \M_k(X)\leq \M_k'(X)$. We will now prove that $\M_k'(X)\leq \cL(X)$ by induction on the number of $k$-faces of $X$. The base case is when the complex has only one $k$-face, {\it i.e.,} the complex is a simplex. In this case $\cL(X)=0=\M_k'(X)$. Since $X$ is $k$-vertex decomposable, \Cref{lem:existance_of_nonconeface} implies that $X_{(k)}^o\neq \emptyset$ and any shedding $k$-face is in $X_{(k)}^o$. Also, since $X$ is $k$-vertex decomposable there exists a  $k$-dimensional shedding face $\sigma$ of $X$ such that $\sigma \in X_{(k)}^o$ and $\del(\sigma, X)$ is a pure $k$-vertex decomposable complex and therefore Cohen-Macaulay. From Proposition \ref{sheddingface} we get that  $\cL(X)\geq \max\{\cL(\del( \sigma,X), \cL(\lk(\sigma,X) + k+1\}$. Thus, from \Cref{definition:theta_k}, we have that 
\begin{equation*}
\begin{split}
    \M_k'(X)& \leq \max\{\M_k'(\del(\sigma,X)),\M_k'(\lk(\sigma,X))+k+1 \}\\
 & \leq \max\{\cL(\del(\sigma,X)),\cL(\lk(\sigma,X))+k+1 \}\\
&\leq \cL(X).
\end{split}    
\end{equation*}
Here, the second inequality follows from induction.
\end{proof}


The following can be easily inferred from \Cref{thm:kvdecomposable} and the fact that $k$-vertex decomposability implies shellability.
\begin{remark}
If $X$ is $k$-dimensional pure complex and $\M_k(X) \neq \C(X)$, then $X$ is not shellable. 
\end{remark}

We now show that the number $d(X, \prec)$ produced by minimal exclusion sequence is also an upper bound for $\M_{0}(X)$. The proof of \cite[Theorem 6]{Lew18} can be modified to show that  $\M_0(X)$ is bounded above by $d(X, \prec)$.

\begin{proposition}\label{thm:theta_mes}
Let $X$ be a simplicial complex, then  $\M_0(X) \leq d(X,\prec) $. 
\end{proposition}
\begin{proof}
The proof is by induction on number of vertices of $X$. If $X$ is a simplex then $\M_0(X)=0 \leq d(X, \prec)$. 
If $X$ is not a simplex and therefore has at least one non-cone vertex $v$, then by definition $\M_0(X) \leq \max\{\M_0(\del(v, X)), \M_0(\lk(v, X))+1\}$. 

The argument in \cite[Theorem 6]{Lew18} shows that $$d(\lk(v, X), \prec)-1 \leq  d(X,\prec) \text{ and }d(\del(v, X), \prec) \leq d(X, \prec).$$ Hence the proof follows by induction.
\end{proof}

By using \Cref{equation:mainresult} and \Cref{thm:theta_mes},  \Cref{thm:bound for nc} can be restated as follows. 

\begin{theorem} \label{thm:mainresultM0}
Let $\H$ be a hypergraph with no isolated vertices. Then
\[C(\rNC(\H))\leq \M_0(\rNC(\H))\leq |V(\H)|- \gamma_i(\H)-1.\]
\end{theorem}

\begin{example}
Let $\H$ be the hypergraph whose edges are the maximal simplices of the complex $\Delta$ given in  \Cref{example1}. Observe that for each vertex $v$ of $\H$,
$\{v\}$ is a dominating set of  $\{1, 2, 3, 4, 5,6 \} \setminus \{v\}$, and the set $\{1, 2, 3, 4, 5, 6\}$ is not an independent set. Thus we conclude that $\gamma_i(\H) = 1$. Since the complement of each edge is an edge in $\H$, we see that the maximal simplices of $\rNC(\H)$ are precisely the edges of $\H$. Hence $\rNC(\H) = \Delta$. 
Therefore, $\M_0(\rNC(\H)) \geq  3$ and from our bound in \Cref{thm:mainresultM0}, we get that $\M_0(\H) \leq 4$.

\end{example}

\section{Future directions}

In \Cref{example1}, we gave a simplicial complex $X$ such that $\M_k(X) < \M_{k-1}(X)$ for $k=1$. However, we are unable to find examples for general $k$. Therefore, we pose the problem here. 

\begin{ques}Given a $k \geq 2$, does there exist a simplicial complex $X$ such that $\M_k(X) < \M_{k-1}(X)$?
\end{ques}

In \Cref{thm:kvdecomposable}, we proved that $\C(X) = \M_k(X)$, if $X$ is $k$-vertex decomposable. It would be interesting to find other classes of simplicial complexes for which $\M_k$ is equal to the collapsibility number. 
\begin{ques}
Classify simplicial complexes $X$ for which $\C(X)=\M_k(X)$.
\end{ques}



\section*{Acknowledgements}  The authors would like to thank Yusuf Civan for his suggestions and for pointing out an error in the earlier preprint of the article. We also thank the Department of Mathematics, Indian Institute of Technology Bombay, for hosting us, where the major part of this work was done. The second author is supported by the seed grant project IITM/SG/SMS/95 by IIT Mandi, India. The Start-up Research Grant SRG/2022/000314 from SERB, DST, India, partially supported the third author.

\bibliography{Collapsibility}
\bibliographystyle{plain}

\end{document}